\documentclass{article}
\usepackage{amsmath}
\usepackage{amsfonts}
\usepackage{amsthm}

\def\congruent{\equiv}

\def\ord{{\rm ord}}
\def\lcm{{\rm lcm}}

\def\notmid{\not\vert\,}

\newtheorem{Theorem}{Theorem}
\newtheorem{Lemma}{Lemma}

\begin{document}

\begin{center}
{\large\bf 
Handling a large bound for a problem on the generalized Pillai equation $\pm r a^x \pm s b^y = c$. 
}

\bigskip

Reese Scott

Robert Styer (correspondence author), Dept. of Mathematical Sciences, Villanova University, 800 Lancaster Avenue, Villanova, PA  19085--1699, phone 610--519--4845, fax 610--519--6928, robert.styer@villanova.edu
\end{center}


revised 19 Dec 2011 

\bigskip

\begin{abstract}  
We consider $N$, the number of solutions $(x,y,u,v)$ to the equation $ (-1)^u r a^x + (-1)^v s b^y = c $
in nonnegative integers $x, y$ and integers $u, v \in \{0,1\}$, for given integers $a>1$, $b>1$, $c>0$, $r>0$ and $s>0$.  Previous work showed that there are nine essentially distinct $(a,b,c,r,s)$ for which $N \ge 4$, except possibly for cases in which the solutions have $r$, $a$, $x$, $s$, $b$, and $y$ each bounded by $8 \cdot 10^{14}$ or $2 \cdot 10^{15}$.  In this paper we show that there are no further cases with $N \ge 4$ within these bounds.  We note that $N = 3$ for an infinite number of $(a,b,c,r,s)$, even if we eliminate from consideration cases which are directly derived from other cases in one of several completely designated ways.  Our work differs from previous work in that we allow $x$ and $y$ to be zero and also allow choices of $(u,v)$ other than $(0,1)$.  
\end{abstract}   

\bigskip

\section{Introduction}
 
The problem of finding $N$, the number of solutions $(x,y,u,v)$ 
to the equation  
\begin{equation}  (-1)^u r a^x + (-1)^v s b^y = c \label{1}  \end{equation}
in nonnegative integers $x, y$ and integers $u, v \in \{0,1\}$,
for given integers $a>1$, $b>1$, $c>0$, $r>0$ and $s>0$, has been considered by many authors with various restrictions on the variables (\cite{Be},  \cite{HT}, \cite{Le}, \cite{Sc-St2}, \cite{Sc-St4}, \cite{Sh}).  See \cite{Be}, \cite{BBM}, \cite{W},  \cite{Sc-St4}, \cite[Section 2]{Sc-St6} for histories of the problem.  

In \cite{Sc-St4} we showed that $N > 3$ implies $\max(a,b,r,s,x,y) < 8 \cdot 10^{14}$ (or, in some cases, $2 \cdot 10^{15}$).  The purpose of this paper is to show that there are exactly nine essentially different cases with $N > 3$ within those bounds.  

To state our main result we need to summarize some definitions from \cite{Sc-St4}.  

We will refer to a {\it set of solutions} to (\ref{1}) which we write as 
$$(a,b,c,r,s;x_1, y_1, x_2, y_2, \dots, x_N, y_N)$$
and by which we mean the set of solutions $(x_1, y_1)$, $(x_2, y_2)$, \dots, $(x_N, y_N)$ to (\ref{1}), with $N > 2$, for given integers $a$, $b$, $c$, $r$, and $s$.  We say that two sets of solutions $(a,b,c,r,s;x_1, y_1, x_2, y_2, \dots, x_N, y_N)$ and $(A,B,C,R,S; X_1, Y_1, X_2, Y_2, \dots, X_N, Y_N)$ belong to the same {\it family} if $a$ and $A$ are both powers of the same integer, $b$ and $B$ are both powers of the same integer, there exists a positive rational number $k$ such that $kc=C$, and for every $i$ there exists a $j$ such that $k ra^{x_i} = R A^{X_j}$ and $ksb^{y_i} = S B^{Y_j}$, $1 \le i,j \le N$.  One can show \cite{Sc-St4} that each family contains a unique member $(a,b,c,r,s;x_1, y_1, x_2, y_2, \dots, x_N, y_N)$ with the following properties: $\gcd(r,s b)=\gcd(s, ra)=1$; $\min(x_1, x_2, \dots, x_N)=\min( y_1, y_2, \dots, y_N)=0$; and neither $a$ nor $b$ is a perfect power.  We say that a set of solutions with these properties is in {\it basic form}.   

The {\it associate} of a set of solutions $(a,b,c,r,s;x_1, y_1, x_2, y_2, \dots, x_N, y_N)$ is the set of solutions $(b,a,c,s,r; \allowbreak y_1, x_1, y_2, x_2, \dots, y_N, x_N)$.  

A {\it subset} of a set of solutions $(a,b,c,r,s;x_1, y_1, x_2, y_2, \dots, x_N, y_N)$ is a set of solutions with the same $(a,b,c,r,s)$ and all it pairs $(x,y)$ among the pairs $(x_i, y_i)$, $1 \le i \le N$.  Note that this subset may be (and, in our usage, usually is) the set of solutions $(a,b,c,r,s;x_1, y_1, x_2, y_2, \dots, x_N, y_N)$ itself.  

We are now ready to state the result of this paper:  

\begin{Theorem} 
\label{Theorem1}
Any set of solutions $(a,b,c,r,s; x_1, y_1, x_2, y_2, \dots, x_N, y_N)$ to (1) with $N > 3$ must be in the same family as a subset (or an associate of a subset) of one of the following: 
\begin{align*}
(3,2,1,1,2; 0,0,1,0,1,1,2,2) \\
(3,2,5,1,2; 0,1,1,0,1,2,2,1,3,4) \\
(3,2,7,1,2; 0,2,2,0,1,1,2,3)\\
(5,2,3,1,2; 0,0,0,1,1,0,1,2,3,6)\\
(5,3,2,1,1; 0,0,0,1,1,1,2,3)\\
(7,2,5,3,2; 0,0,0,2,1,3,3,9)\\
(6,2,8,1,7; 0,0,1,1,2,2,3,5)\\
(2,2,3,1,1; 0,1,0,2,1,0,2,0)\\
(2,2,4,3,1; 0,0,1,1,2,3,2,4)
\end{align*}
\end{Theorem}

Note that there are an infinite number of cases with $N=3$, even if we consider only sets of solutions in basic form (see (\ref{62}) through (\ref{69}) in Section 2).

\section{Preliminary Results}  

If, for a given choice of $(a,b,c,r,s)$, (\ref{1}) has two solutions $(x_i, y_i)$ and $(x_j, y_j)$, $1 \le i,j \le N$, we have 
\begin{equation} r a^{\min{(x_i, x_j)}}(a^{|x_j - x_i|} + (-1)^\gamma ) = s b^{\min(y_i, y_j)} (b^{|y_j - y_i|} + (-1)^\delta )  \label{2} \end{equation}
where $\gamma, \delta \in \{0,1\}$.  

In the following lemma, we summarize some results which follow from the proof of Theorem 2 of \cite{Sc-St4}.  

\begin{Lemma} 
\label{SummaryLemma}
Any set of solutions violating Theorem 1 above must be in the same family as a basic form (or the associate of a basic form) which satisfies one of the following:
\begin{equation} (ra,sb)=1, \max(a,b,r,s, x, y) < 8 \cdot 10^{14}, 0 = x_1 < x_2 < x_3< \dots < x_N, 0 = y_1 < y_2 < y_3 < \dots < y_N, \label{19b} \end{equation}   
\begin{equation} (ra,sb)=1, \max(a,b,r,s, x, y) < 8 \cdot 10^{14}, 0 = x_1 < x_2 < x_3 < \dots < x_N, 0 = y_2 < y_1 < y_3 < \dots < y_N, \label{21b} \end{equation}
\begin{equation} (ra,sb)=1, \max(a,r,s, x, y) < 8 \cdot 10^{14}, 0 = x_1 < x_2 < x_3 < \dots < x_N, 0 = y_1 = y_2 < y_3 < \dots < y_N, \label{20b} \end{equation}      
\end{Lemma} 

We will also need the following two lemmas from \cite{Sc-St4}.  

\begin{Lemma} (Lemma 13 of \cite{Sc-St4})  
\label{ZBoundLemma}
Suppose $(ra,sb)=1$ and suppose (\ref{1}) has four solutions $(x_1, y_1)$, $(x_2, y_2)$, $(x_3, y_3)$, $(x_4, y_4)$ with $x_1 < x_2 < x_3 < x_4$.  Let $Z = \max(x_4, y_1, y_2, y_3, y_4)$.   Then
$$a^{x_3-x_2} \le Z, s \le Z+1. $$
\end{Lemma}

\begin{Lemma} (Lemma 17 of \cite{Sc-St4})  
\label{SigmaLemma}
Let $a>1$ and $b>1$ be relatively prime integers.  For $1 \le i \le m$, let $p_i$ be one of the $m$ distinct prime divisors of $a$.  Let $p_i^{g_i} || b^{n_i} \pm 1$, where $n_i$ is the least positive integer for which there exists a positive integer $k$ such that $\vert b^{n_i} - k p_i \vert = 1$, and $\pm$ is read as the sign that maximizes $g_i$.  Write
$$\sigma = \sum_i g_i \log(p_i)/ \log(a).$$
Then, if 
$$a^x \mid b^y \pm 1, $$
where the $\pm$ sign is independent of the above, we must have 
$$a^{x} \mid  a^{\sigma} y.$$
\end{Lemma}

Define $\sigma_a(b)$ to be the $\sigma$ of this lemma, and let $\sigma_b(a)$ be the $\sigma$ of this lemma with the roles of $a$ and $b$ reversed.

In the course of this paper, we will often need to show that a given set of three solutions does not have a fourth solution.  We eliminate the possibility of a fourth solution by one of three methods:  using the method known as \lq bootstrapping' (see  \cite{GLS} or \cite{St1}), using bounds derived from LLL basis reduction (see \cite{Sm}), or using $x_4$ to calculate $y_4$ and seeing if $y_4$ is an integer (or using $y_4$ to calculate $x_4$). 

The technique known as \lq bootstrapping' assumes one knows the values of $a$, $b$, $r$, $s$, $x_3$, and $y_3$.  For simplicity of exposition assume $\gamma=\delta=1$ (the other cases are only slightly more complicated) and consider (\ref{2}) with $(i,j)=(3,4)$:
$$ r a^{x_3} (a^{x_4-x_3} -1) = s b^{y_3} (b^{y_4-y_3} -1). $$  
Let $\ord(n,p)$ be the least positive integer such that $p \mid n^{\ord(n,p)} -1$.  Since $(ra,sb)=1$, for each prime $p \mid s b^{y_3}$, we have $\ord(a, p) \mid x_4-x_3$.  Let $x_0 = \lcm \{ \ord(a,p) : p \mid sb^{y_3} \}$ so $x_0 \mid x_4-x_3$.  Similarly, define $y_0 = \lcm \{ \ord(b,p) : p \mid r a^{x_3} \}$ so $y_0 \mid y_4-y_3$.  Now we begin the bootstrapping steps.  For each prime $p \mid a^{x_0} -1$ such that $p \notmid sb$, $\ord(b,p)$ must divide $y_4 - y_3$; setting $y_0 = \lcm( y_0, \{ \ord(b,p) : p \mid a^{x_0}-1 \})$, we have this new $y_0 \mid y_4 - y_3$.  For each prime $p \mid b^{y_0} -1$ such that $p \notmid ra$, $\ord(a,p) \mid x_4-x_3$; set $x_0 = \lcm( x_0,  \{ \ord(a,p) : p \mid b^{y_0} -1 \}) $, so this new $x_0 \mid x_4-x_3$.  We alternately use $x_0$ to find a larger $y_0$ if possible, the new $y_0$ to find a larger $x_0$ if possible, etc., continuing to bootstrap back and forth until $x_0$ or $y_0$ exceeds $8 \cdot 10^{14}$, in which case $x_4-x_3$ or $y_4-y_3$ must exceed this bound, contradicting Lemma~\ref{SummaryLemma}, so no fourth solution exists.  

For the LLL basis reduction algorithm, we follow the exposition in \cite{Sm}.  We have $\pm c = r a^{x_4} - s b^{y_4}$ so
$${ r a^{x_4} \over s b^{y_4} } = 1 \pm {c \over s b^{y_4} }$$
and thus
\begin{equation} \log\left( {r \over s} \right) + x_4 \log(a) - y_4 \log(b) = \log\left( 1 \pm { c \over s b^{y_4} } \right). \label{22} \end{equation}
Since $ \vert \log(1 \pm x) \vert < 2 x $ for $0 < x < 0.5$, one derives from (\ref{22}) that, if $ c/(sb^{y_4}) < 0.5$,  
$$\log\left( {r \over s} \right) + x_4 \log(a) - y_4 \log(b) < 2 { c \over s b^{y_4} } $$
so 
$$\log\left( {r \over s} \right) + x_4 \log(a) - y_4 \log(b) <  { 2 c \over s} e^{-\log(b) y_4} . $$  
Choosing $C= 10^{36}$, we set 
$$A = \left( 
\begin{array}{cc}
1 & [C \log(a)] \\
0 & [-C \log(b)]  
\end{array}
\right),
Y = \left( 
\begin{array}{cc}
 0 & \left[- C \log\left( { r \over s} \right) \right]
\end{array}
 \right),
$$
where we set $[X]$ to be the closest integer to $X$.  Note that rows in the computer algebra program Maple correspond to columns in \cite{Sm}.  Let $B$ be the LLL basis reduction of the rows of $A$, and let $b_1$ and $b_2$ be the rows of $B$.  Set $b_2^* = b_2 - {b_2 \cdot b_1 \over b_1 \cdot b_1} b_1$.  Further, define the vector $\sigma = Y B^{-1}$, define the number $\sigma_2 = \sigma[2]$, and let $\{\sigma_2\} = \sigma_2 - [\sigma_2]$ be the distance from $\sigma_2$ to the nearest integer.  In our context, Lemma VI.1 in \cite{Sm} becomes

\begin{Lemma} 
Let $S= \left( 8 \cdot 10^{14} \right)^2$, $T= 8 \cdot 10^{14} + 0.5$, $c_1 = \max(1, ||b_1||^2 / ||b_2^*||^2)$, $c_2 = {2 r \over s}$, $c_3 = \log(b)$, and $c_4 = c_1^{-1} \{\sigma_2\} ||b_1||^2$.  Assume $ c/(sb^{y_4}) < 0.79$.  If $c_4^2 > S + T^2$ then 
\begin{equation} y_4 \le { 1 \over c_3} \left( \log(C c_2) - \log\left( \sqrt{c_4^2 - S} - T \right) \right).  \label{25} \end{equation}
\end{Lemma}

Given $a$, $b$, $c$, $r$, $s$, and verifying that $c/(s b^{y_4}) < 0.5$, we can often use this lemma to find that $\max(x_4,y_4) < \min(a,s-1)$, 
contradicting Lemma~\ref{ZBoundLemma}, and so we can conclude that no fourth solution exists.

Finally, the third useful method to eliminate the possibility of a fourth solution assumes we are given $a$, $b$, $r$, $s$, and either a potential $x_4$ or a potential $y_4$.    Suppose we have a bound $c < 10^{1000}$ and suppose $ y_4 \ge b \ge 1000$ so $b^{y_4} > 10^{3000}$.   
Using $\vert \log(1 \pm x) \vert < 2x$ for $x<0.5$, we see that (\ref{22}) implies 
\begin{equation*} \left\vert \log\left( {r \over s} \right) + x_4 \log(a) - y_4 \log(b) \right\vert < 10^{-1500}. \end{equation*}
So we have to more than 1000 places of accuracy,  
\begin{equation} y_4 = x_4 { \log(a) \over \log(b) } + { \log(r/s) \over \log(b) }  \label{solvey4eqn} \end{equation}
and 
\begin{equation} x_4 = y_4 { \log(b) \over \log(a) } + { \log(s/r) \over \log(a) }.  \label{solvex4eqn} \end{equation} 
If we know $a$, $b$, $r$, $s$, and $x_4$, we can calculate $y_4$ from (\ref{solvey4eqn}) and if this $y_4$ is not an integer to 1000 places of accuracy, then no fourth solution can exist.  Similarly, given $y_4$ we can calculate $x_4$ from (\ref{solvex4eqn}) and hope that $x_4$ is not an integer, thus showing no fourth solution can exist.

We will also need the following.  

\begin{Lemma} 
\label{NewLemma6}
Suppose (\ref{1}) has three solutions $(x_1, y_1, u_1, v_1)$, $(x_2, y_2, u_2, v_2)$, and $(x_3, y_3, u_3, v_3)$ and further assume that the following four conditions hold:

1.) \quad $x_1 < x_2 < x_3$ and $y_1 < y_2 < y_3$,

2.)  \quad $u_1 \ne v_1$,

3.)  \quad  any solution $(x,y)$ to (\ref{1}) such that $x > x_1$ and $y > y_1$ must also satisfy $x \ge x_2$ and $y \ge y_2$,

4.)  \quad  $R = {r a^{x_1} \over \gcd(r a^{x_1} , s b^{y_1})} > 2$ and $S = {s b^{y_1} \over \gcd(r a^{x_1} , s b^{y_1})} > 2$. 

Then $x_2 - x_1 \mid x_3 - x_1$ and $y_2 - y_1 \mid y_3 - y_1$.  
\end{Lemma}  

\begin{proof}  
Suppose we have three solutions to (\ref{1}) satisfying all four conditions of the lemma.  Considering (\ref{2}) with $(i,j)=(1,2)$ and $(1,3)$ we have 
$$R (a ^{x_2 - x_1} +(-1)^{\gamma_2} ) = S (b^{y_2 - y_1} + (-1)^{\delta_2})$$
and
$$R (a ^{x_3 - x_1} +(-1)^{\gamma_3}) = S (b^{y_3 - y_1} + (-1)^{\delta_3}).$$
Since $u_1 \ne v_1$, we must have $\gamma_2 = \delta_2$ and $\gamma_3 = \delta_3$.  Let $\alpha = 1$ if $\gamma_2 = \gamma_3 = 1$, otherwise let $\alpha=0$.  
Let
$$t = {a^{x_2-x_1} +(-1)^{\gamma_2} \over S } = {b^{y_2 - y_1} + (-1)^{\gamma_2} \over R }$$
and
$$T = {a^{x_3-x_1} +(-1)^{\gamma_3} \over S } = {b^{y_3 - y_1} +(-1)^{\gamma_3} \over R }.$$
Note that $t$ and $T$ are both integers.  

Let $g_1 = \gcd(x_2 - x_1, x_3-x_1)$ and $g_2 = \gcd(y_2-y_1, y_3-y_1)$.  
Let $k$ be the least integer such that $b^k +(-1)^{\alpha}$ is divisible by $R$.  Then $k$ must divide both $y_2-y_1$ and $y_3-y_1$, so that $k$ divides $g_2$, and 
$$b^{g_2} +(-1)^{\alpha} = R l_2$$
for some integer $l_2$. 
(Note that, when $\alpha=0$, $2^n || k$ implies $2^n || y_2 - y_1$ when $\gamma_2 = 0$ and $2^{n+1} \mid y_2 - y_1$ when $\gamma_2 = 1$, similarly for $y_3 - y_1$,  so that, since $\min(\gamma_2, \gamma_3) = 0$, we have $2^n || g_2$.)    
Similarly, 
$$a^{g_1} +(-1)^\alpha = S l_1$$ 
for some integer $l_1$.  
Since $g_1$ divides both $x_2-x_1 $ and $x_3 -x_1$, $l_1$ divides $t$ and $T$.  
There must be an integer $j$ which is the least positive integer such that $b^j +(-1)^\alpha$ is divisible by $R l_1$, and $j$ must divide both $y_2-y_1$ and $y_3-y_1$, so that $j$ divides $g_2$.  Therefore, $l_1 | l_2$.

A similar argument with the roles of $a$ and $b$ reversed shows that $l_2 | l_1$, so that $l_1 = l_2$, and we have 
\begin{equation} r a^{x_1} (a^{g_1} +(-1)^{\alpha})= s b^{y_1} (b^{g_2} +(-1)^{\alpha}). \label{16} \end{equation}
(\ref{16}) shows that $(x_1+g_1, y_1+g_2)$ is a solution to (\ref{1}).  If $x_1+g_1 \ne x_2$, then, using Condition 3 in the formulation of the lemma, we see that we must have $x_1+g_1>x_2$, which is impossible by the definition of $g_1$.   So $x_1+g_1=x_2$ and, similarly, $y_1+g_2=y_2$. 
\end{proof}

When $N=3$, we find many sets of solutions.  Here we list several types of sets of solutions, each one of which generates an infinite number of basic forms (and therefore an infinite number of families) giving three solutions to (\ref{1}).  We list these sets of solutions in the form $(a,b,c,r,s;x_1, y_1, x_2, y_2, x_3, y_3)$:  

\begin{equation} (a, {a^{kd} + (-1)^{u+v} \over a^d + (-1)^{u}}, { a^d b + (-1)^{u+v+1} \over h}, {b+(-1)^v \over h}, {a^d+(-1)^{u} \over h}; 0, 1, d, 0, kd, 2) \label{62}  \end{equation}
where $a$ and $b = {a^{kd} + (-1)^{u+v} \over a^d + (-1)^{u}}$ are integers greater than 1, $d$ and $k$ are positive integers, $h=\gcd(a^d+(-1)^u, b+(-1)^v)$, and $u$ and $v$ are in the set $\{ 0,1 \}$.  When $u=0$, we take $k-v$ odd; when $(u,v) = (1,1)$, we take $a^d \le 3$.  When $a=d=2$ and $(u,v)=(1,1)$, we can take $k$ to be a half integer.  When $k=2$ and $u-v$ is odd, the same choice of $(a,b,r,s)$ as in (\ref{62}) gives the additional set of solutions 
\begin{equation}  (a, a^d +(-1)^v, {2 a^d + (-1)^v \over h} ,{a^d +(-1)^v 2 \over h}, {a^d + (-1)^{v+1} \over h}; 0,0,d,1,3d,3).  \label{63}  \end{equation}

Other sets of solutions can be constructed with specified values of $a$.  When $a=3$ we have 
\begin{equation}  (3, {3^g +(-1)^v \over 2}, {3^{g+1} + (-1)^{v}  \over 2^{2+v-\alpha} }, { 3(3^{g-1}+(-1)^{v}) \over 2^{2+v-\alpha}}, 2^{1-v+\alpha}; 0,1, 1,0,2g,3) \label{64}  \end{equation}
where $v \in \{0,1\}$, $g$ is a positive integer, $\alpha = 0$ when $2 \mid g-v$, $\alpha=1$ when $g$ is odd and $v=0$, and $\alpha = 2$ when $g$ is even and $v=1$.    

When $a=2$, we have 
\begin{equation}  (2, 2^g + (-1)^v, 2^{g} +(-1)^{v+1}, 2, 1; 0,1,g-1,0,g,1) \label{65}  \end{equation}
where $v \in \{0,1\}$ and $g$ is a positive integer.   

Also, it is easy to construct sets of solutions for which $x_1 = y_1 = y_2 = 0$.  For example, we have, for $a$ even and $x>0$, 
\begin{equation}  (a, 2 a^x \pm 1, a^x \pm 1, 2, a^x \mp 1; 0,0,x,0,2x,1). \label{66}   \end{equation}
More generally,
\begin{equation}  (a, b , { a^{x_2} +(-1)^{t} \over 2^{m} } , 2^{1-m}, { a^{x_2} +(-1)^{t+1} \over 2^m }; 0,0,x_2, 0, x_3, y_3), \label{67}  \end{equation}
with $b^{y_3} = {2a^{x_3} +(-1)^{t+w+1}a^{x_2} + (-1)^{w+1} \over  a^{x_2} + (-1)^{t+1} }$, where $x_2>0$, $x_3>0$, $x_2 \mid x_3$, and $a^{x_3} \congruent (-1)^w \bmod {a^{x_2} +(-1)^{t+1} \over 2^m}$, for $t \in \{ 0,1 \}$, $w \in \{0,1\}$, and $m=1$ or 0 according as $a$ is odd or even. 

We also find an infinite family for which $\gcd(a,b) > 1$:
\begin{equation} (a,b,c,r,s; x_1, y_1, x_2, y_2, x_3, y_3) = (a, ta, {a (t+(-1)^{u+v+1}) \over h} , { ta + (-1)^{v} \over h}, { a +(-1)^{u} \over h}; 0,0,1,1,m+1,2) \label{68} \end{equation}
where $m \ge 0$ is an integer, $t = {a^m + (-1)^{v} \over a+(-1)^u}$ is an integer, $h= \gcd( ta + (-1)^v, a+(-1)^u)$, and $u$ and $v$ are in the set $\{ 0,1 \}$. 
Closely related to (\ref{68}) is the following:
\begin{equation} (2, 4t, {4t+4 \over h_1} , { 4t+1 \over h_1}, { 3 \over h_1}; 0,0,2,1,m_1+2,2) \label{69} \end{equation} 
where $m_1 \ge -1$ is an odd integer, $t = {2^{m_1} +1 \over 3}$, $h_1 = 3$ or 1 according as $m_1 \congruent 5 \bmod 6$ or not, and $u,v \in \{0,1\}$.

The Maple\texttrademark\ worksheets with the calculations for the following sections can be found at \cite{St2}.

\section{ Case (\ref{19b}) $0 = x_1 < x_2 < x_3 < x_4$, $0 = y_1 < y_2 < y_3 < y_4$}    

\begin{Lemma}  
\label{Lemma5}
Suppose $a$ or $b \le 170000$.  Then (\ref{1}) satisfying conditions (\ref{19b}) has at most three solutions.   
\end{Lemma}

\begin{proof}  
Suppose (\ref{19b}) holds.  By symmetry, we may assume $a>b$.  Further assume $a$ is not a perfect power.  Taking (\ref{2}) with $(i,j)=(2,3)$ we have $r a^{x_2} (a^{x_3-x_2} +(-1)^\gamma) = s b^{y_2} (b^{y_3-y_2} +(-1)^\delta)$.  Lemma~\ref{ZBoundLemma} shows $b^{y_3-y_2} \le Z$ so Lemma~\ref{SummaryLemma} shows $y_3 - y_2 < \log(8 \cdot 10^{14}) / \log(b)$.  For each choice of $b \le 170000$, $\delta \in \{0,1\}$, and $y_3 - y_2 < \log(8 \cdot 10^{14}) / \log(b)$, we see that $a^{x_2}$ is a factor of $b^{y_3-y_2} +(-1)^\delta < 8 \cdot 10^{14}$; this is small enough to factor easily, so we can list all factors $a^{x_2}$, hence we know all possible $a$ and $x_2$.  Lemma~\ref{ZBoundLemma} bounds $a^{x_3 - x_2} \le Z$.  For each choice of $x_3 - x_2 < \log(8 \cdot 10^{14}) / \log(a)$ and of $\gamma \in \{0,1\}$, we calculate the power of $b$ dividing $a^{x_3-x_2} +(-1)^\gamma$ which gives us the maximal possible value for $y_2$, call it $y_{2,max}$.  If $y_{2, max}>0$, for each $y_2$ with $1 \le y_2 \le y_{2,max}$ and $h = \gcd(a^{x_3-x_2} +(-1)^\gamma, b^{y_3-y_2} +(-1)^\delta)$, we can solve for $r = (b^{y_3-y_2} +(-1)^\delta)/(a^{x_2} h)$ and $s = (a^{x_3-x_2} +(-1)^\gamma)/(b^{y_2} h)$.  We have $y_3 = (y_3 - y_2) + y_2$ and $x_3 = (x_3 - x_2) + x_2$.  We see that $c = | r a^{x_2} - s b^{y_2}|$.  We now check if $r (a^{x_2} +(-1)^\alpha) = s (b^{y_2} +(-1)^\beta)$ for some $\alpha, \beta \in \{0,1\}$.   If so, we have values $a$, $b$, $c$, $r$, and  $s$ for which (\ref{1}) has at least three solutions.  

For each set of at least three solutions found for $b \le 170000$, we use the LLL method or bootstrapping to show that there is no fourth solution.
\end{proof}

\begin{Lemma}  
\label{Case63Lemma}
No instance of (\ref{63}) with $d=1$ has a fourth solution.  
\end{Lemma}

\begin{proof}
By Lemma~\ref{Lemma5}, we only need to consider $a > 170000$.

Assume (\ref{1}) has three solutions satisfying (\ref{63}) with $d=1$.  If (\ref{1}) has a further solution $(x,y)$ with $y=2$, then 
$$ h r a^{x} = \pm h c \pm h s b^2 = \pm (2 a +(-1)^v ) \pm (a+(-1)^{v+1} ) (a + (-1)^v)^2 \not\equiv 0 \bmod a^2, $$
contradicting Lemma~\ref{SummaryLemma}, 
which requires $x \ge 2$.  So if $d=1$, any fourth solution $(x_4, y_4)$ to (\ref{63}) must satisfy $x_4 > 3$ and $y_4 > 3$.   

Lemma~\ref{ZBoundLemma} shows that $\max(x_4, y_4) \ge a \ge 170000$ while Lemma~\ref{SummaryLemma} shows that $ c  \le r + s < 16 \cdot 10^{14}$ so certainly $x_4 > 100$.  

When $d=1$ and $v=0$, we have $b=a+1$, $rh = a+ 2$, $sh = a -1$, and $ch = 2 a +  1$, where $h = \gcd(a+2, a-1)$.  
Considering the solution $(x_3, y_3)=(3,3)$, we get
\begin{equation}  (2a+1) + (a-1) b^3 = (a+2) a^3.  \label{MSA} \end{equation} 
Considering the solution $(x_4, y_4)$ we get
\begin{equation}  (2a+1) + (a-1) b^{y_4} \congruent 0 \bmod (a+2) a^{100}. \label{MSB} \end{equation}  
Combining (\ref{MSA}) and (\ref{MSB}) we find
\begin{equation} (a-1)(b^{y_4-3} -1) \congruent 0 \bmod (a+2) a^3,  \label{MSC} \end{equation}  
which requires $a^2 \mid y_4 - 3$ except possibly when $a \congruent 2 \bmod 4$.  If $a \congruent 2 \bmod 4$, let $2^g \mid \mid a+2 = b+1$ and let $a/2 = a_0$; then instead of (\ref{MSC}) we can use 
$$ (b^{y_4-3} -1) \congruent 0 \bmod 2^{g+3} a_0^3 $$
which again requires $a^2 \mid y_4 -3$.  So we can write $y_4  = 3 + j a^2$ for some integer $j > 0$.
For some integers $M_1$ and $M_2$ we have
\begin{align*}
&(2 a + 1) + (a-1) \Bigg( 1 + (3 + j a^2) a + \frac{(3+ j a^2)(3+ j a^2-1)}{ 2} a^2 \\
& \qquad + \frac{(3+ j a^2) ( 3+ j a^2 -1) (3+ j a^2 -2)}{ 6} a^3 \\
& \qquad + \frac{(3+ j a^2) ( 3+ j a^2 -1) (3+ j a^2 -2)(3+ j a^2 -3)}{ 24} a^4 + M_1 a^5 \Bigg) = M_2 a^{100}
\end{align*}
from which we derive, for some integer $M_3$, 
\begin{equation}  (2- j) a^3 + \left(1 - \frac{3}{2} j \right) a^4 =  \frac{M_3}{6} a^5 \label{MSD}  \end{equation}  
so that $j = {wa}/{6} + 2$ for some integer $w \ge 0$.  If $w=0$ then (\ref{MSD}) becomes impossible.  If $w>0$ then $y_4 \ge  3 + \left(\frac{a}{6} + 2 \right) a^2 > 8 \cdot 10^{14}$ when $a > 170000$, so we have a contradiction to Lemma~\ref{SummaryLemma}, showing the impossibility of a fourth solution.   

When $d=1$, $v=1$, we have $b=a-1$, which is equivalent to the previous case after reversing the roles of $a$ and $b$.  Thus, in every case, no instance of (\ref{63}) has a fourth solution.  
\end{proof}

\begin{Lemma} 
(\ref{1}) satisfying (\ref{19b}) has at most three solutions.    
\end{Lemma}

\begin{proof}
Using Lemma~\ref{Lemma5}, assume $a>b>170000$.  
Suppose (\ref{1}) satisfying (\ref{19b}) has four solutions with $0 = x_1 < x_2 < x_3 < x_4$ and $0 = y_1 < y_2 < y_3 < y_4$.  To fix notation, let 
$$c = -(-1)^\alpha r + (-1)^\beta s = (-1)^\gamma (r a^{x_2} - s b^{y_2}) =  (-1)^\delta (r a^{x_3} - s b^{y_3}) = (-1)^\epsilon (r a^{x_4} - s b^{y_4})$$ 
for some $\alpha, \beta, \gamma, \delta, \epsilon \in \{0,1\}$.
Considering (\ref{2}) with $(i,j)=(1,2)$ and $(i,j)=(2,3)$, we have 
\begin{equation} r(a^{x_2} +(-1)^{\alpha+\gamma}) = s(b^{y_2} + (-1)^{\beta+\gamma}) \label{106a}  \end{equation}
and 
\begin{equation} r a^{x_2} (a^{x_3-x_2} - (-1)^{\gamma+\delta}) = s b^{y_2} (b^{y_3 - y_2} - (-1)^{\gamma+\delta}). \label{106b} \end{equation}

Suppose $x_3-x_2 = y_3-y_2 = 1$.  Since $(ra,sb)=1$, from (\ref{106b}) we have $a^{x_2} \mid b - (-1)^{\gamma+\delta}$; therefore, since $a>b$, we have $x_2=1$, $(-1)^{\gamma+\delta} = -1$, and $a=b+1$.  Then $b^{y_2} \mid a -(-1)^{\gamma+\delta} = a+1 = b+2$ which is impossible for $b>2$.  Thus we have $\max(x_3-x_2, y_3-y_2) \ge 2$.  By Lemma~\ref{ZBoundLemma}, $\max(x_3-x_2, y_3- y_2) \le 2$ since $a>b>170000$.  Thus,
\begin{equation}  \max(x_3-x_2, y_3-y_2) = 2.  \label{XX}  \end{equation}

Taking the ratio of (\ref{106a}) and (\ref{106b}), 
\begin{equation} { a^{x_2} (a^{x_3-x_2} -(-1)^{\gamma+\delta}) \over a^{x_2}+(-1)^{\alpha+\gamma} } = { b^{y_2} (b^{y_3-y_2} -(-1)^{\gamma+\delta}) \over b^{y_2} +(-1)^{\beta +\gamma} }.  \label{107} \end{equation}
We rewrite this as 
$$ (a^{x_3-x_2} -(-1)^{\gamma+\delta}) \left( 1 - {(-1)^{\alpha+\gamma} \over a^{x_2} + (-1)^{\alpha+\gamma} } \right) =  (b^{y_3-y_2} -(-1)^{\gamma+\delta}) \left( 1 - { (-1)^{\beta+\gamma} \over b^{y_2} +(-1)^{\beta+\gamma} } \right).$$
From this we obtain
$$a^{x_3 - x_2} - b^{y_3 - y_2} = a^{x_3 - x_2} {(-1)^{\alpha+\gamma} \over a^{x_2} + (-1)^{\alpha+\gamma} } -  b^{y_3 - y_2} { (-1)^{\beta+\gamma} \over b^{y_2} +(-1)^{\beta+\gamma} } 
- {(-1)^{\alpha+\delta} \over a^{x_2} + (-1)^{\alpha+\gamma} }  + { (-1)^{\beta+\delta} \over b^{y_2} +(-1)^{\beta+\gamma} } $$
so
$$\vert a^{x_3 - x_2} - b^{y_3 - y_2} \vert \le   a^{x_3 - x_2} {1 \over a^{x_2} + (-1)^{\alpha+\gamma} } +  b^{y_3 - y_2} { 1 \over b^{y_2} +(-1)^{\beta+\gamma} } 
+ {1 \over a^{x_2} + (-1)^{\alpha+\gamma} }  + { 1 \over b^{y_2} +(-1)^{\beta+\gamma} }.  $$ 
Since $a^{x_2} +(-1)^{\alpha+\gamma} \ge b$ and $b^{y_2} +(-1)^{\beta+\gamma} \ge b-1$,  
$$\vert a^{x_3 - x_2} - b^{y_3 - y_2} \vert \le {a^{x_3-x_2} \over b} + {b^{y_3-y_2} \over b-1} + {1 \over b} + {1 \over b-1} $$
and since (\ref{XX}) gives  $\max(a^{x_3-x_2}, b^{y_3-y_2}) \ge b^2$,
$$\vert a^{x_3 - x_2} - b^{y_3 - y_2} \vert < \max(a^{x_3-x_2}, b^{y_3-y_2}) \left( {1 \over b} + {1 \over b-1} + {1 \over b^3} + {1 \over (b-1) b^2} \right).$$
 Therefore,  
$$\vert a^{x_3 - x_2} - b^{y_3 - y_2} \vert  <  { 2 \over b-1} \max \left( a^{x_3-x_2}, b^{y_3-y_2} \right)  $$
so that, since $b > 3$,
\begin{equation} \left( 1 - {2 \over b-1} \right) b^{y_3-y_2} < a^{x_3-x_2} < \left(1 - {2 \over b-1} \right)^{-1} b^{y_3-y_2}  \label{109} \end{equation}
and similarly 
\begin{equation}  \left( 1 - {2 \over b-1} \right) a^{x_3-x_2} < b^{y_3 - y_2} < \left( 1 - {2 \over b-1} \right)^{-1} a^{x_3-x_2}. \label{110}  \end{equation}

We now show $x_2 \le x_3 - x_2$.  From (\ref{106b}) and (\ref{110}) we have 
$$a^{x_2} \mid b^{y_3-y_2} - (-1)^{\gamma+\delta} < \left( 1 - {2 \over b-1} \right)^{-1} a^{x_3-x_2} + 1 <  a^{x_3-x_2+1},$$
so $x_2 \le x_3-x_2$.  Similarly, one can show that $y_2 \le y_3 - y_2$. 

We can further show $x_2 < x_3-x_2$.  Suppose $z=x_2 = x_3-x_2$.  Then (\ref{106b}) gives $a^z \mid b^{y_3-y_2} -(-1)^{\gamma+\delta}$ so (\ref{110}) shows
$$a^z \mid b^{y_3-y_2} -(-1)^{\gamma+\delta}  <  \left( 1 - {2 \over b-1} \right)^{-1} a^z +1 < 2 a^z.$$
Thus, $a^z = b^{y_3-y_2} -(-1)^{\gamma+\delta}$ and (\ref{107}) becomes 
$$(b^{y_3-y_2} - (-1)^{\gamma+\delta} 2) (b^{y_2} +(-1)^{\beta+\gamma}) = b^{y_2} (b^{y_3-y_2} +(-1)^{\alpha+\gamma} - (-1)^{\gamma+\delta})$$
which is impossible modulo $b$.  Thus, $1 \le x_2 < x_3 - x_2$; similar arguments show $1 \le y_2 < y_3-y_2$.  
 
Recalling (\ref{XX}) we find $x_2=y_2=1$ and $x_3 = y_3 = 3$.  If $(-1)^{\gamma+\delta} = 1$ then (\ref{107}) can be rewritten as $a (a -(-1)^{\alpha+\gamma}) = b (b - (-1)^{\beta+\gamma})$.  This implies $a=b -(-1)^{\beta+\gamma}$ and so $b = a - (-1)^{\alpha+\gamma}$, $\beta \ne \gamma$ and $\alpha = \gamma$, so $c = -(-1)^\alpha r + (-1)^\beta s > 0$ shows $\alpha=\gamma= 1$ and $\beta=0$.  Then $r=(a-2)/h$, $s=(a+1)/h$, and $c= (2a-1)/h$ where $h = \gcd(a-2, a+1) \le 3$.  
We see that the case under consideration in this paragraph satisfies (\ref{63}) with $d=v=1$.  By Lemma~\ref{Case63Lemma}, this cannot lead to a fourth solution.    

So $(-1)^{\gamma+\delta} = -1$, and (\ref{107}) becomes 
$$a (a^2 +1) (b +(-1)^{\beta+\gamma}) = b (b^2 +1) (a+(-1)^{\alpha+\gamma}).$$
Since $\gcd(a^2 +1, a \pm 1) \le 2$ and $\gcd(b^2+1, b \pm 1) \le 2$, we must have 
\begin{equation}  a +(-1)^{\alpha+\gamma} \mid 2 (b +(-1)^{\beta+\gamma}), b+(-1)^{\beta+\gamma} \mid 2(a+(-1)^{\alpha+\gamma}).  \label{110b} \end{equation}
Note that (\ref{109}) gives 
$$a < { b \over  \sqrt{ 1 - \frac{2}{b-1} } }  < b+2 $$
for $b \ge 6$ so $a=b+1$.  But then (\ref{110b}) is impossible.  
\end{proof}

\section{Case (\ref{21b}) $0 = x_1 < x_2 < x_3 < x_4$, $0 = y_2 < y_1 < y_3 < y_4$}  

We begin with some preliminaries.  To fix notation, let 
\begin{equation} c = (-1)^\alpha r + s b^{y_1}= r a^{x_2} +(-1)^\beta s = (-1)^\gamma (r a^{x_3} - s b^{y_3}) = (-1)^\delta (r a^{x_4} - s b^{y_4}) \label{110c} \end{equation}
for some $\alpha$, $\beta$, $\gamma$, and $\delta \in \{ 0,1\}$.  
  Applying (\ref{2}) with $(i,j)=(2,3)$, $(1,2)$, and $(1,3)$, we have 
\begin{equation}  r  ( a^{x_3} - (-1)^{\gamma} a^{x_2}) = s (b^{y_3} +(-1)^{\beta+\gamma}), \label{111a}  \end{equation}
\begin{equation}  r (a^{x_2} - (-1)^{\alpha} ) = s (b^{y_1} - (-1)^{\beta}),  \label{111b} \end{equation}
and
\begin{equation} r (a^{x_3} - (-1)^{\alpha+\gamma}) = s ( b^{y_3} +(-1)^\gamma b^{y_1} ).  \label{111c} \end{equation}

Taking the ratio of (\ref{111a}) and (\ref{111b}), we obtain
\begin{equation} {  a^{x_3} - (-1)^{\gamma} a^{x_2} \over a^{x_2} - (-1)^{\alpha} } = { b^{y_3} +(-1)^{\beta+\gamma} \over  b^{y_1} - (-1)^{\beta} } . \label{111}  \end{equation}
Similarly, the ratio of (\ref{111c}) with (\ref{111b}) gives
\begin{equation}  { a^{x_3}  - (-1)^{\alpha+\gamma} \over a^{x_2} - (-1)^{\alpha} } = { b^{y_3} +(-1)^{\gamma} b^{y_1}  \over b^{y_1} - (-1)^{\beta} }, \label{112}  \end{equation}
and considering (\ref{111a}) with (\ref{111c}), one gets 
\begin{equation}   { a^{x_3}  - (-1)^{\gamma} a^{x_2} \over a^{x_3} - (-1)^{\alpha+\gamma} } = { b^{y_3} +(-1)^{\beta+\gamma}  \over b^{y_3} + (-1)^{\gamma} b^{y_1} }. \label{113}  \end{equation}

\begin{Lemma}  
\label{SigmaBoundLemma}
For $b \le 10^3$ and $a < 8 \cdot 10^{14}$, $b^{\sigma_b(a)} < 10^{22}$.
\end{Lemma}

\begin{proof}
Let $b=p$ be an odd prime less than 1000.  Let $k = \lceil 22 \log(10) / \log(b) \rceil $, so $p^{k} \ge 10^{22}$.  If $p^{\sigma_p(a)} \ge 10^{22}$ then $\sigma_p(a) \ge k$.  Choose $\alpha \in \{0,1\}$ to minimize the positive integer $n$ such that $p \mid a^n +(-1)^\alpha$.  Note that $n \mid (p-1)/2$.  By definition, $p^{\sigma_p(a)} \mid a^n +(-1)^\alpha$.  Let $a_0$ be any solution to $x^n +(-1)^\alpha \congruent 0 \bmod p$.  Using Hensel's lifting lemma, we find a unique solution $a_1$ to the congruence  $a_1^n +(-1)^\alpha \congruent 0 \bmod p^{k}$ with $a_1 \congruent a_0 \bmod p$.  For each prime $p<10^3$, for each $n | (p-1)/2$, and for each solution $a_0 \bmod p$, calculations show that the associated $a_1 \bmod p^{k}$ exceeds $8 \cdot 10^{14}$.  In other words, for every prime $p < 10^3$, if  $p^{\sigma_p(a)} \ge 10^{22}$, then $a > 8 \cdot 10^{14}$.  

Now suppose $b = p_1^{\beta_1} p_2^{\beta_2}$.  By definition of $\sigma$, $b^{\sigma_b(a)} = p_1^{g_1} p_2^{g_2}$ for some positive integers $g_1$ and $g_2$.  
For any given $1 \le k_2 \le 22 \log(10)/ \log(p_2)$, let
$$k_1 = \left\lceil { 22 \log(10) - k_2 \log(p_2) \over \log(p_1)} \right\rceil.$$
Note that if $b^{\sigma_b(a)} = p_1^{g_1} p_2^{g_2} \ge 10^{22}$ for some $a$, then there exists $1 \le k_2 \le 22 \log(10) / \log(p_2)$ such that $k_2 \le g_2$ and $k_1 \le g_1$.  

Suppose $b^{\sigma_b(a)} = p_1^{g_1} p_2^{g_2}$ with $k_1 \le g_1$ and $k_2 \le g_2$.  By definition of $\sigma$, there exists $n_1 \mid (p_1 - 1)/2$, $n_2 \mid (p_2 - 1)/2$, and $\alpha_1, \alpha_2 \in \{0,1\}$ such that
$$p_1^{g_1} \mid a^{n_1} +(-1)^{\alpha_1}, p_2^{g_2} \mid a^{n_2} +(-1)^{\alpha_2},$$
so
$$a^{n_1} +(-1)^{\alpha_1} \congruent 0 \bmod p_1^{k_1}, a^{n_2} +(-1)^{\alpha_2} \congruent 0 \bmod p_2^{k_2}.  $$  
For each choice of $n_1 \mid (p_1 - 1)/2$, we can list all values of $a \bmod p_1$ with $a^{n_1} +(-1)^{\alpha_1} \congruent 0 \bmod p_1$.  Using Hensel's lifting lemma, we can obtain a complete list of all possible values of $a \bmod p_1^{k_1}$ satisfying $a^{n_1} +(-1)^{\alpha_1} \congruent 0 \bmod p_1^{k_1}$.  Similarly, for each choice of $n_2 \mid (p_2 - 1)/2$, we can obtain a complete list of all possible values of $a \bmod p_2^{k_2}$ satisfying $a^{n_2} +(-1)^{\alpha_2} \congruent 0 \bmod p_2^{k_2}$.  Thus, for each choice of $p_1$, $p_2$, $k_2$, $n_1 \mid (p_1 -1)/2$, $n_2 \mid (p_2 -1)/2$, and $\alpha_1, \alpha_2 \in \{0,1\}$, we can obtain every possible $a \bmod p_1^{k_1} p_2^{k_2}$ satisfying 
$$p_1^{k_1} \mid  a^{n_1} +(-1)^{\alpha_1}, p_2^{k_2} \mid a^{n_2} +(-1)^{\alpha_2}.$$
Calculations show that each potential $a$ exceeds $8 \cdot 10^{14}$.  In other words, if $b^{\sigma_b(a)} \ge 10^{22}$ then $a > 8 \cdot 10^{14}$.  

Now suppose $b = p_1^{\beta_1} p_2^{\beta_2} p_3^{\beta_3}$ with $p_1 < p_2 < p_3$.  By definition of $\sigma$, $b^{\sigma_b(a)} = p_1^{g_1} p_2^{g_2} p_3^{g_3}$ for some positive integers $g_1$,  $g_2$, and $g_3$.  
For any given $k_3 \le 22 \log(10)/ \log(p_3)$, and any $k_2 \le ( 22 \log(10) - k_3 \log(p_3)) / \log(p_2)$, let
$$k_1 = \left\lceil { 22 \log(10) - k_3 \log(p_3) - k_2 \log(p_2) \over \log(p_1)} \right\rceil.$$
If $b^{\sigma_b(a)} = p_1^{g_1} p_2^{g_2} p_3^{g_3} \ge 10^{22}$ for some $a$, there exist $k_3$ and $k_2$ such that $k_3 \le  g_3$, $k_2 \le  g_2$, and $k_1 \le g_1$. 
For each choice of $p_1$, $p_2$, $p_3$, $k_3$, $k_2$, $n_1 \mid (p_1 -1)/2$, $n_2 \mid (p_2 -1)/2$, $n_3 \mid (p_3 -1)/2$, $\alpha_1, \alpha_2, \alpha_3 \in \{0,1\}$,  we proceed as in the previous paragraph to obtain every possible $a \bmod p_1^{k_1} p_2^{k_2} p_3^{k_3}$ satisfying 
$$p_1^{k_1} \mid  a^{n_1} +(-1)^{\alpha_1}, p_2^{k_2} \mid a^{n_2} +(-1)^{\alpha_2}, p_3^{k_3} \mid a^{n_3} +(-1)^{\alpha_3}.$$
The calculations verify that each possible $a$ exceeds $8 \cdot 10^{14}$.

Finally, suppose $b < 10^3$ is the product of four primes.  The same procedure works.  Since each $b \le 10^3$ has four or fewer distinct prime factors, we conclude that if $b^{\sigma_b(a)} \ge 10^{22}$ then $a > 8 \cdot 10^{14}$.   
\end{proof}

The following two lemmas apply to the following set of solutions:  
\begin{equation}  (a,b,c,r,s; x_1, y_1, x_2, y_2, \dots, x_N, y_N), N \ge 3, 0= x_1 < x_2 < x_3 < \dots <  x_N, 0= y_2 < y_1 < y_3 < \dots < y_N. \label{401}  \end{equation}

\begin{Lemma}  
\label{x2y1BoundLemma}  
If (\ref{401}) holds, then either $x_2 \le x_3-x_2$ or $y_1 \le y_3-y_1$.  
\end{Lemma}

\begin{proof}
Assume the set of solutions $(a,b,c,r,s; x_1, y_1, x_2, y_2, \dots, x_N, y_N)$ satisfies (\ref{401}), and further assume $x_2>x_3-x_2$ and $y_1 > y_3-y_1$.  

From (\ref{111}) we get 
\begin{equation}
 a^{x_3-x_2} - (-1)^\gamma + { (-1)^\alpha a^{x_3-x_2} - (-1)^{\alpha + \gamma} \over a^{x_2} - (-1)^\alpha } = b^{y_3-y_1} + { (-1)^\beta b^{y_3-y_1} + (-1)^{\beta+\gamma}  \over b^{y_1} -(-1)^\beta }. 
\label{401.5}
\end{equation}
If $\min(a,b)>2$, we see that $\vert { (-1)^\alpha a^{x_3-x_2} - (-1)^{\alpha + \gamma} \over a^{x_2} - (-1)^\alpha } \vert \le { 3 + 1 \over 9 - 1} = {1 \over 2}$.  Similarly, $\vert { (-1)^\beta b^{y_3-y_1} + (-1)^{\beta+\gamma}  \over b^{y_1} -(-1)^\beta } \vert  \le {1 \over 2}$.  In both cases the value $1/2$ is possible only when $a$ (respectively, $b$), equals 3, $x_2$ (respectively, $y_1$) equals 2, and $x_3-x_2$ (respectively, $y_3-y_1$) equals 1.  By Lemma 3 of \cite{Sc-St4}, 
$(a,b)=1$, so we must have
\begin{equation} \vert a^{x_3-x_2} - b^{y_3-y_1} -(-1)^\gamma \vert  < 1, \label{402}  \end{equation}
so the left side of (\ref{402}) must be zero.  

But now from (\ref{113}) we have 
\begin{equation}  (-1)^\gamma a^{x_3-x_2} - (-1)^\gamma b^{y_3-y_1} -1 = { (-1)^{\beta+\gamma} a^{x_3} - (-1)^{\alpha+\gamma} b^{y_3} -(-1)^{\alpha+\beta} \over a^{x_2} b^{y_1} } . \label{403}  \end{equation}  
But, from (\ref{402}), the left side of (\ref{403}) must be zero, which is impossible since the numerator on the right side of (\ref{403}) cannot be zero by Mihailescu's theorem \cite{Mih} since $x_3$ and $y_3$ both are greater than 2.  

So we can assume $\min(a,b) = 2$.  We see that $\vert { (-1)^\alpha a^{x_3-x_2} - (-1)^{\alpha + \gamma} \over a^{x_2} - (-1)^\alpha } \vert \le { 2 + 1 \over 4 - 1} = 1$.  Similarly, we have $\vert { (-1)^\beta b^{y_3-y_1} + (-1)^{\beta+\gamma}  \over b^{y_1} -(-1)^\beta } \vert  \le 1$.  In both cases the value $1$ is possible only when $a$ (respectively, $b$), equals 2, $x_2$ (respectively, $y_1$) equals 2, and $x_3-x_2$ (respectively, $y_3-y_1$) equals 1.  $(a,b)=1$, so we must have
\begin{equation} \vert a^{x_3-x_2} - b^{y_3-y_1} -(-1)^\gamma \vert  < 2, \label{X82}  \end{equation}
so the left side of (\ref{X82}) must be zero or one.  If the left side of (\ref{X82}) is zero, then again we can use (\ref{113}) to obtain a contradiction as above.  If the left side of (\ref{X82}) equals one, then recall $(a,b) =1$ and note that $a^{x_3-x_2} - b^{y_3-y_1} = \pm 2$ is impossible when $\min(a,b)=2$. 
\end{proof}

\begin{Lemma}  
\label{Lemma11}
If $\min(a,b)>6$ in (\ref{401}), then $x_2 = x_3 - x_2$ implies $y_1 \le y_3 - y_1$, and also $y_1 = y_3-y_1$ implies $x_2 \le x_3 - x_2$.  
\end{Lemma}

\begin{proof}  
By symmetry, it suffices to prove that $x_2 = x_3 - x_2$ implies $y_1 \le y_3 - y_1$.  Assume $y_1 > y_3-y_1$ and $x_2 = x_3-x_2$.  Since $\frac{b^{y_3-y_1} + (-1)^\gamma }{b^{y_1} - (-1)^\beta} < 1$, we cannot have $\alpha = \gamma$ in (\ref{401.5}), so that $\frac{a^{x_3-x_2} - (-1)^\gamma}{a^{x_2} - (-1)^\alpha} = 1 \pm \frac{2}{A \mp 1}$ where $A= a^{x_2}= a^{x_3-x_2}$.  Since $\frac{2}{A \mp 1} +  \frac{b^{y_3-y_1} + (-1)^\gamma }{b^{y_1} - (-1)^\beta} < 1$, we must have $\frac{b^{y_3-y_1} + (-1)^\gamma }{b^{y_1} - (-1)^\beta} = \frac{2}{A \mp 1}$, so that $b^{y_3-y_1} + (-1)^\gamma \mid 2(b^{y_1} -(-1)^\beta)$.  Then, using the elementary divisibility properties of $b^y \pm 1$ (for general integer $y$), we see that, since $b > 3$, we must have $y_3 - y_1 \mid y_1$.  Let $B = b^{y_3-y_1}$.  Let $j = \frac{2 y_1 - y_3}{y_3-y_1}$, noting that $j$ is a positive integer.  Then 
\begin{equation*} \frac{1 + \frac{(-1)^\gamma }{B} }{ B^j \left(1 - \frac{(-1)^\beta}{B^{j+1} } \right) } = \frac{2}{A(1 \mp \frac{1}{A})}, \end{equation*}
so that, letting $k = \min(a,b) >6$,
\begin{equation}  B \le \frac{A ( 1 + \frac{1}{k})^2 }{2(1-\frac{1}{k^2})} = \frac{A(1+\frac{1}{k})}{2(1-\frac{1}{k})} \le \frac{2}{3} A.  \label{403.5} \end{equation} 
But from (\ref{401.5}) we get 
$$ B \ge A - 2 = A (1 - \frac{2}{A}) \ge A(1-\frac{2}{k} ) \ge \frac{5}{7} A, $$
contradicting (\ref{403.5}).
\end{proof}

\begin{Lemma} \label{Small21Lemma}  
Suppose $a, b \le 1000$.  Then (\ref{1}) satisfying conditions (\ref{21b}) has at most three solutions.  
\end{Lemma}  

\begin{proof}
Assume (\ref{21b}) holds.  By symmetry, we may assume $a>b$.  Let $ b \le 10^3$ and assume $a$ is not a perfect power.    Considering (\ref{2}) with $(i,j)=(3,4)$ and applying Lemma~\ref{SigmaLemma} and Lemma~\ref{SigmaBoundLemma}, we have 
\begin{equation} b^{y_3} | b^{\sigma_b(a)} (x_4 - x_3)<  10^{22} Z < 8 \cdot 10^{36}, \label{35new} \end{equation}
so $y_3 < \log(8 \cdot 10^{36})/\log(b)$.  

Choose $b \le 1000$ and $\nu \in \{0,1\}$.  Considering (\ref{111a}) and noting that $(ra,sb)=1$, $a^{x_2}$ must be a divisor of $b^{y_3} +(-1)^\nu$.  It is easy to factor $b^{y_3} +(-1)^\nu < 8 \cdot 10^{36}$, so for each $b$, $y_3$, and $\nu$ we obtain every possible $a^{x_2}$ hence a complete list of possible values for $a$ and its associated $x_2$ exponent.  Lemma~\ref{ZBoundLemma} gives a bound $x_3-x_2 < \log(8 \cdot 10^{14}) / \log(a)$.  For each $x_3-x_2$ within this bound and each $\mu \in \{0,1\}$, we can solve $r= (b^{y_3} +(-1)^{\nu})/( a^{x_2} h)$ and $s = (a^{x_3-x_2} +(-1)^\mu)/ h$ where $h = \gcd(a^{x_3-x_2} +(-1)^\mu, b^{y_3} +(-1)^{\nu})$.   

Considering (\ref{2}) with $(i,j)=(1,3)$, we should have $r (a^{x_3} +(-1)^\eta) = s b^{y_1} (b^{y_3 - y_1} + (-1)^\theta)$ for some $\eta$, $\theta \in \{0,1\}$.  We now determine if there is a value $\eta \in \{0,1\}$ for which $b | a^{x_3} +(-1)^\eta$, in which case we determine $y_1$ such that $b^{y_1} || a^{x_3} + (-1)^\eta$.  Now we see if there is a value $\theta \in \{0,1\}$ for which $r  (a^{x_3} +(-1)^\eta) = s b^{y_1} (b^{y_3 - y_1} + (-1)^\theta)$.  If so, we have three solutions to (\ref{1}) with $c = | r a^{x_3} - s b^{y_3}|$.  

Now apply bootstrapping as outlined above; our calculations show that in each case $x_4$ or $y_4$ exceeds $8 \cdot 10^{14}$, hence there is no fourth solution.     
\end{proof}

We can reformulate (\ref{62}) as 
\begin{equation} ({ b^{kd} +(-1)^{u+v} \over b^d +(-1)^u } ,b, {a b^d - (-1)^{u+v} \over h}  , {b^d+(-1)^u \over h}, {a +(-1)^v \over h} ; 0,d,1,0,2,kd), \label{10a} \end{equation}
where $h= \gcd(a+(-1)^v,  b^d +(-1)^u)$, $d$ and $k$ are positive integers, and if $u=0$ then $k-v$ must be odd, or if $u=1$ then $v=0$.

\begin{Lemma} 
\label{Case62Lemma}
No member of the infinite class (\ref{10a}) with $b > 3$ has a fourth solution.    
\end{Lemma}

\begin{proof}
Suppose (\ref{1}) has three solutions satisfying (\ref{10a}).  If $r+s b^d = r a + s = c$, then $u=v=1$, contradicting the conditions of (\ref{10a}).  So we must have either 
\begin{equation}  s b^d - r  = c \label{A} \end{equation} 
or 
\begin{equation}  r a - s  = c. \label{B} \end{equation}
Assume (\ref{10a}) has a fourth solution $(x_4, y_4)$.  By Lemma~\ref{SummaryLemma}, we can assume (\ref{21b}) holds, so that
\begin{equation}  x_4 > 2, y_4 > kd.  \label{E} \end{equation}  
By Lemma~\ref{Small21Lemma}, we may assume $b > 1000$.    
Now $a \ge b^{(k-1) d} - b^{(k-2)d} > 8 \cdot 10^{14}$ when $b > 1000$ and $(k-1)d \ge 5$, so, by Lemma~\ref{SummaryLemma}, we must have $(k-1)d < 5$.  
For all $d$ and $k$ satisfying this bound, for $u, v \in \{0,1\}$, and for $1000 < b \le 6000$, we use the LLL and bootstrapping methods to show that no instance of (\ref{10a}) with $ 1000 < b \le 6000$ can have a fourth solution. 
Now, using $b>6000$, we have $(k-1)d < 4$.    
In particular, the only pairs $(d,k)$ that we need to consider are $(3,2)$, $(2,2)$, $(1,4)$, $(1,3)$, and $(1,2)$.  

Note that $k=2$ implies $b^d - (-1)^u = a$ so that, when $k=2$ we must have $h \le 3$ and therefore, since we have shown $d>1$ implies $k=2$, we certainly have
\begin{equation}  d > 1 {\rm \ implies \  } \min(r,s)>2.  \label{C} \end{equation}

We will now use Lemma~\ref{NewLemma6} to show $d \mid y_4$.  When (\ref{A}) holds, we apply Lemma~\ref{NewLemma6} to the solutions $(0, d)$, $(2, kd)$, and $(x_4, y_4)$, noting that (\ref{E}) gives Conditions (1.) and (3.) of Lemma~\ref{NewLemma6}, (\ref{A}) gives Condition (2.), and, if we assume $d>1$ then (\ref{C}) gives Condition (4.)  (recall (\ref{21b}) requires $(ra,sb)=1$).  Now we can use Lemma~\ref{NewLemma6} to get $(k-1)d \mid y_4 - d$.  Similarly, when (\ref{B}) holds, we can apply Lemma~\ref{NewLemma6} to the solutions $(1,0)$, $(2, kd)$, and $(x_4, y_4)$ to get $kd \mid y_4$.  In either case, we have $d \mid y_4$.  So, in considering the set of solutions $(a,b,c,r,s; 0,d,1,0,2,kd,x_4, y_4)$, we can assume $d=1$ without loss of generality, noting that we have reformulated the meanings of $b$ and $y_4$.   

When $d=1$ and $k=2$, $3$, or $4$, we can apply the LLL basis reduction method to show that there are no solutions for $b \le 6 \cdot 10^5$.  There are a handful of $b$ values for which the LLL method fails, and we handle these by bootstrapping. 

Using $a \ge b-1$ and $\max(x_4, y_4) \ge a$ we see that $y_4 = \max(x_4, y_4)$ implies $y_4 \ge a \ge b-1$, while $x_4 = \max(x_4, y_4)$ implies $a^a - (a (a+1) + 1) \le r a^{x_4} - c \le s b^{y_4} \le (a+1)^{y_4 + 1}$.  In either case, we certainly have $y_4 > 100$.

Suppose $d=1$ and $k =4$.  Then $a \ge b^3 - b^2 + b -1 > 8 \cdot 10^{14}$ for $b > 6 \cdot 10^5$ so this case cannot lead to a fourth solution.    

Suppose $d=1$ and $k=3$, so $a = b^2 \mp b + 1$, $rh=b \pm 1$, $sh= b^2 \mp b + 2$, and $ch = b^3 \mp b^2 + b \mp 1$, where we take the upper sign when $u=v=0$ and the lower sign when $u=1$ and $v=0$.  
From (\ref{10a}) we have
\begin{equation}  r a^{x_4} + c \congruent r a^2 + c \congruent 0 \bmod s b^3.  \label{re1} \end{equation}
When $b$ is even, $h$ is odd and $s$ is even, so 
\begin{equation}  r a^{x_4} + c \congruent r a^2 + c \congruent 0 \bmod 2 b^3.  \label{re2} \end{equation}  
When $b \congruent 2 \bmod 4$, then, letting $2^g || a+1 = sh$, we have, noting $2 \notmid h$, 
\begin{equation}  r a^{x_4} + c \congruent r a^2 + c \congruent 0 \bmod 2^{g+3} (b/2)^3.  \label{re3} \end{equation}
Let $n$ be the least number such that $a^n \congruent 1 \bmod G$ where $G = b^3$ when $b$ is odd, $G = 2 b^3$ when $b \congruent 0 \bmod 4$, and $G = 2^{g+3} (b/2)^3$ when $b \congruent 2 \bmod 4$.  Then, using the elementary divisibility properties of $a^x - 1$ (for general integer $x$), we find that when $b$ is odd, $n=b^2$, and when $b$ is even, $n = 2 b^2$.  Now from (\ref{re1}), (\ref{re2}), and (\ref{re3}) we see that we have $x_4 = 2 + j b^2$ where $j \ge 0$ is an integer.  (We note that when $b$ is even, $j$ is even, although for this case we will not need this.)  
We have, for some integer $M$,  
$$ a^{x_4} = (b^2 \mp b + 1)^{x_4} = 1 + x_4 (b^2 \mp b) + \frac{ x_4 (x_4 -1) }{2} (b^2 \mp b)^2 + \frac{x_4 (x_4-1)(x_4-2)}{6} (b^2 \mp b)^3 + M (b^2 \mp b)^4 $$
and thus
\begin{align*} 
rh a^{x_4} + ch = (b \pm 1)  \Bigg( 1 + & (2 + j b^2) (b^2 \mp b) + \frac{(2 + j b^2)(1+j b^2)}{2} (b^2 \mp b)^2 + \frac{(2+ j b^2)(1+j b^2) j b^2}{6} (b^2 \mp b)^3   \\
  &  + M (b^2 \mp b)^4 \Bigg)  +(b^3 \mp b^2 + b \mp 1). \\
\end{align*}
Collecting like powers we find (recalling $y_4 > 100$)
$$ rh a^{x_4} + ch = (2-j) b^3 + \frac{ M_1 }{6} b^4 = M_2 b^{100}  $$
where $M_1$ and $M_2$ are integers.  Thus, $j = wb/6 + 2$ for some integer $w \ge 0$, so that, when $w>0$, $x_4 \ge 2 + 2 b^2 + b^3 /6 > 8 \cdot 10^{14}$ since $b>6 \cdot 10^5$, contradicting Lemma~\ref{SummaryLemma}.  So $w=0$ and $x_4 = 2 + 2 b^2 < 8 \cdot 10^{14}$, so that $b < 2 \cdot 10^7$.  We apply (\ref{solvey4eqn}) for each $b$ with $6 \cdot 10^5 < b \le 2 \cdot 10^7$.  The calculations show that for every $b$ in this range, (\ref{solvey4eqn}) never gives an integral value for $y_4$ within 25 places of accuracy, hence $x_4 = 2 + 2 b^2$ cannot lead to a fourth solution.  Thus, when $d=1$ and $k=3$, (\ref{10a}) cannot have a fourth solution.

Suppose $d=1$, $k=2$, $u=1$, and $v=0$, so $a = b + 1$, $rh=b-1$, $sh= b + 2$, and $ch = b^2 + b + 1$.  
From (\ref{10a}) we have $r a^{x_4} + c \congruent r a^2 + c \congruent 0 \bmod s b^2$.  Now proceeding as in the case $k=3$, we derive $x_4 = 2 + j b$ where $j$ is even when $b$ is even.  Write $x_4 = 2 + e_1 b+ e_2 b^2$ with $0 \le e_1 < b$, noting $2 | e_1$ when $2 | b$.  For some integer $M$ we have 
$$ rh a^{x_4} + ch = (b -1) (1 + (2+e_1 b+ e_2 b^2) b + \frac{ (2+e_1 b+ e_2 b^2) (1+e_1 b+ e_2 b^2)}{2} b^2 + M b^3) + (b^2+b+1).$$
Noting $\big( x_4 (x_4 - 1)/2 \big)  - 1 $ is an integer divisible by $b$, and noting $y_4 > 100$, we have, for some integers $M_1$ and $M_2$,
$$ rh a^{x_4} + ch = (2-e_1) b^2 + M_1 b^3 = M_2 b^{100} $$
so that $e_1 \congruent 2 \bmod b$, so $e_1 = 2$.

Write $x_4 = 2 + 2 b + e_2 b^2$ with $0 \le e_2$.  For some integer $M$ we have
\begin{align*}  rh a^{x_4} + ch = (b -1) & \Bigg(1 + (2+ 2 b+ e_2 b^2) b + \frac{ (2+ 2 b+ e_2 b^2) (1+ 2 b+ e_2 b^2)}{2} b^2     \\
& + \frac{(2+2b+e_2 b^2)(1+2b+ e_2 b^2) (2b+e_2 b^2)}{6} b^3 + M b^4 \Bigg)+(b^2+b+1).  \\
\end{align*}
From this we obtain, for some integers $M_1$ and $M_2$, 
$$  rh a^{x_4} + ch = -e_2 b^3 + \frac{M_1}{6} b^4 = M_2 b^{100}.$$
Thus $e_2 = wb/6$ for some integer $w \ge 0$.  As before, we show $w=0$, so it remains only to deal with $x_4 = 2 + 2b$, in which case (\ref{solvey4eqn}) implies 
\begin{align}
y_4 &= {1 \over \log(b)} \left( x_4 \log(b+1) - \log(b+2) + \log(b-1) \right) \notag \\
&= {1 \over \log(b)} \left( x_4 \log(b) + x_4 \log(1+ 1/b) - \log(1 + 2/b) + \log(1-1/b) \right). 
\end{align}
From this one gets
$$y_4 - x_4 = {1 \over \log(b) } \left( x_4 \log\left(1 + {1 \over b}\right) - \log\left(1+ {2 \over b}\right) + \log\left(1 - {1 \over b}\right) \right). $$
Plugging in $x_4 = 2+2b$ into $x_4 \log\left(1 + {1 \over b}\right) - \log\left(1+ {2 \over b}\right) + \log\left(1 - {1 \over b}\right)$ and taking Taylor series in terms of $z=1/b$, one can show that the Taylor series $2 - 2z + 7z^2/6 - 17 z^3/6 + \dots$ is alternating, hence, 
$$ 0< {1 \over \log(b) } \left( 2 - 2 {1 \over b}\right)  <  y_4-x_4  < {2 \over \log(b)} < 1$$
which contradicts $y_4$ being an integer.  So $x_4 = 2 + 2 b$ is not possible in this case, completing the proof that no fourth solution is possible when $d=1$, $k=2$, $u=1$, $v=0$.

Suppose $d=1$, $k=2$, $u=0$, and $v=1$.  Then $a=b-1$ (recall we have reformulated $b$ to represent $b^{d}$), and reversing the roles of $a$ and $b$, this becomes identical to the case just completed.    

Thus (\ref{10a}) does not have a fourth solution. 
\end{proof}

\begin{Lemma} 
\label{Large21Lemma}
(\ref{1}) satisfying conditions (\ref{21b}) has at most three solutions.  
\end{Lemma}

\begin{proof}
By Lemma~\ref{Small21Lemma} we can assume $a>b>1000$.  
Suppose (\ref{1}) has four or more solutions and (\ref{21b}) holds.   

From (\ref{111}) we get
$$ a^{x_3-x_2}  + {(-1)^{\alpha} a^{x_3-x_2} - (-1)^{\gamma} a^{x_2} \over a^{x_2} - (-1)^{\alpha} }  = b^{y_3-y_1}  + { (-1)^{\beta} b^{y_3-y_1} +(-1)^{\beta+\gamma} \over b^{y_1} - (-1)^{\beta} } $$
and so
\begin{equation} \vert  a^{x_3-x_2} - b^{y_3 - y_1} \vert \le { a^{x_3-x_2 } \over a^{x_2} - 1 } + 1 + { 1 \over a^{x_2} - 1 } +  { b^{y_3-y_1}  \over b^{y_1} - 1 } + {1 \over b^{y_1} - 1 } . \label{115} \end{equation}
Then the right side of (\ref{115}) can be bounded as follows: 
\begin{equation} \vert  a^{x_3-x_2} - b^{y_3 - y_1} \vert \le { a^{x_3-x_2} \over b} + 1 + { 1 \over b } +  { b^{y_3-y_1}  \over b-1} + {1 \over b-1 }, \label{116} \end{equation}
and noting that $\max(a^{x_3-x_2}, b^{y_3-y_1}) \ge b$, 
$$ \vert  a^{x_3-x_2} - b^{y_3 - y_1} \vert \le \max(a^{x_3-x_2}, b^{y_3-y_1}) \left( {1 \over b} + { 1 \over  b}  + { 1 \over b^2 } + { 1 \over b-1} + {1 \over (b-1) b} \right) $$
so 
$$ \vert  a^{x_3-x_2} - b^{y_3 - y_1} \vert < \max(a^{x_3-x_2}, b^{y_3-y_1}) \left( {3 \over b-1} \right). $$
Since $b>4$, 
\begin{equation} \left( \left(1- { 3 \over b-1} \right) b^{y_3-y_1} \right)^{{1 \over x_3-x_2}} < a < \left( \left(1- { 3 \over b-1}\right)^{-1}  b^{y_3-y_1} \right)^{{1 \over x_3-x_2}}.  \label{117}  \end{equation}
Since $a>b$, this implies $x_3-x_2 \le y_3 -y_1$.

Since $1000 < b$ and $b^{y_3-y_1} < 8 \cdot 10^{14}$ we must have $y_3-y_1 \le 4$. 
We divide this proof into a number of cases:
$$2 \le x_3-x_2 < y_3-y_1 \le 4,$$
$$3 \le x_3-x_2 = y_3 - y_1 \le 4$$
$$x_3-x_2=y_3-y_1=2$$
$$x_3-x_2=1, y_1 \ge y_3 - y_1$$
$$x_3-x_2=1, y_1 < y_3-y_1 = 3, 4$$
$$x_3-x_2=1, 1= y_1 < y_3-y_1=2$$

In each case, a key idea is to use (\ref{115}).  

If $2 \le x_3-x_2 \le y_3-y_1 \le 4$, noting that $\max(a^{x_3-x_2}, b^{y_3-y_1}) \ge b^2$, (\ref{116}) implies
$$ \vert  a^{x_3-x_2} - b^{y_3 - y_1} \vert \le \max(a^{x_3-x_2}, b^{y_3-y_1}) \left( {1 \over b} + { 1 \over  b^2}  + { 1 \over b^3 } + { 1 \over b-1} + {1 \over (b-1) b^2} \right) $$
so
$$ \vert  a^{x_3-x_2} - b^{y_3 - y_1} \vert  < \max(a^{x_3-x_2}, b^{y_3-y_1}) \left( {2 \over b-2} \right). $$
Since $b>4$, 
\begin{equation}  \left( \left(1- { 2 \over b-2} \right) b^{y_3-y_1} \right)^{{1 \over x_3-x_2}} < a < \left( \left(1- { 2 \over b-2}\right)^{-1}  b^{y_3-y_1} \right)^{{1 \over x_3-x_2}}.  \label{119} \end{equation}

Suppose $2 \le x_3-x_2 < y_3-y_1 \le 4$.  Given $y_3-y_1$ and $x_3-x_2$ with these bounds, consider each $b$ with $1000 < b < ( 8 \cdot 10^{14})^{{1 \over y_3-y_1}}$.  Then for each $a$ satisfying (\ref{119}) with $\gcd(a,b)=1$, we consider (\ref{2}) with $(i,j)=(3,4)$ and apply Lemma~\ref{SigmaLemma} to get $b^{y_3} | b^{\sigma_b(a)} (x_4 - x_3)<  b^{\sigma_b(a)} 8 \cdot 10^{14}$.  Thus, $y_3 < \sigma_b(a) + \log(8 \cdot 10^{14})/\log(b)$.  From (\ref{111a}) $a^{x_2} \mid  b^{y_3} +(-1)^{\beta+\gamma}$.  For each $y_3 < \sigma_b(a) + \log(8 \cdot 10^{14})/\log(b)$ and each $\beta$, $\gamma \in \{ 0,1 \}$, let $x_{2, max}$ be the largest power of $a$ dividing $b^{y_3} +(-1)^{\beta+\gamma}$.  If $x_{2,max} >0$ then for each $1 \le x_2 \le x_{2,max}$ and each $\alpha \in \{0,1\}$, we set $y_1 = y_3 - (y_3-y_1)$, $r = (b^{y_1} - (-1)^{\beta})/h$, and $s =(a^{x_2} - (-1)^{\alpha} )/h$ where $h = \gcd(a^{x_2} - (-1)^{\alpha} , b^{y_1} - (-1)^{\beta})$.  Let $c = (-1)^\alpha r + s b^{y_1}$.  If $c= r a^{x_2} +(-1)^\beta s = (-1)^\gamma (r a^{x_3} - s b^{y_3})$ then we have three solutions to (\ref{1}).   Our calculations show that for $(x_3-x_2, y_3-y_1) = (2,3)$, $(2,4)$, or $(3,4)$, only two closely related cases of three solutions occur: 
\begin{align*} 
(a,b,c,r,s; x_1, y_1, x_2, y_2, x_3, y_3) =& (56744, 1477, 83810889, 1478, 56743; 0,1,1,0,3,4), \\
    &  (56745, 1477, 41906182, 739, 28373; 0,1,1,0,3,4). \\
\end{align*} 
We apply bootstrapping to these; our calculations show that either there is no fourth solution or else $y_4 > 8 \cdot 10^{14}$, impossible, so these two sets of three solutions do not extend to a fourth solution.

Suppose $x_3-x_2=y_3-y_1=z$ where $z=3$ or 4.  From (\ref{119}) we obtain the impossibility   
\begin{equation}  b < a < \left( 1 - {2 \over b-2} \right)^{ {-1 \over z} } b \le \left( 1 + \frac{2}{b-4} \right)^{1/3} b < b+1. \label{120} \end{equation}

Suppose $x_3-x_2=y_3-y_1=2$.  We consider several subcases. 

Suppose $ x_2 \ge x_3-x_2=2$ and $y_1 \ge y_3-y_1=2$.  Then, since $b>4$, (\ref{115}) implies the impossibility 
$$a^2 - b^2 < 1 + {1 \over b} +1 + {1 \over b} + 1 + {1 \over b-1} + { 1 \over b-1} <4.$$

Suppose $x_2 \ge x_3 - x_2=2$, $1 = y_1 < y_3-y_1 = 2$.  Then, since $b>4$, (\ref{115}) implies the impossibility 
$$ 2b+1 \le  a^2 - b^2 < 1 + {1 \over b} +1 + {1 \over b} + {b^2 \over b-1} + { 1 \over b-1} < b+4.$$

Suppose $1 = x_2 < x_3 -x_2 = 2$, $y_1 \ge y_3-y_1=2$.  Then, since $b>4$, (\ref{115}) implies the impossibility
$$2a-1 \le   a^2 - b^2  < {a^2 \over a -1} + 1 + {1 \over a-1} + 1 + {1 \over b-1} + {1 \over b-1} < a+4.$$

Suppose $1 = x_2 < x_3-x_2=2$, $1 = y_1 < y_3-y_1=2$.  Then (\ref{119}) shows that
$$b < a < {b \over \sqrt{1 - {2 \over b-2} } } < b \left( 1 + {1 \over b-4} \right) < b+2$$
for $b>8$ so $a=b+1$.  Now (\ref{111}) implies
\begin{equation}  (b+1) (  (b+1)^2 - (-1)^{\gamma}) (b -(-1)^\beta) - ( b+1 - (-1)^{\alpha}) ( b^{3} +(-1)^{\beta+\gamma}) = 0. \label{120b} \end{equation}
Expanding in powers of $b$, one gets a polynomial $q_3 b^3 + q_2 b^2 + q_1 b + q_0$ where each coefficient satisfies $ | q_i | \le 7$ and $| q_2 | > 0$.  No such polynomial can have a zero for $b>1000$, hence (\ref{120b}) is impossible, so this case cannot lead to a fourth solution.  

So we may now assume $x_3-x_2 = 1$.

Consider first the case $x_3-x_2=1$ and $y_1 \ge y_3 -y_1$.  By Lemmas \ref{x2y1BoundLemma} and \ref{Lemma11}, we can assume $y_1 = y_3-y_1$, so that, by Lemma~\ref{Lemma11} $x_2 = x_3-x_2 = 1$, in which case (\ref{401.5}) is possible only when $\alpha=\gamma \ne \beta$, so that we have (\ref{10a}) with $k=2$, $u=\alpha$, and $v=\beta$, which has no fourth solution by Lemma~\ref{Case62Lemma}.

Consider next the case $x_3-x_2=1$ and $y_1 < y_3-y_1$.  
Since $x_2 \ge x_3-x_2 = 1$, (\ref{115}) can be bounded as
$$ \vert  a - b^{y_3 - y_1} \vert \le 1 + { 1 \over b} + 1 + { 1 \over b } +   {b^{y_3-y_1} \over b-1} + {1 \over b -1 }. $$
Since $y_1 < y_3-y_1$, we have 
$$ \vert  a - b^{y_3 - y_1} \vert \le b^{y_3-y_1} \left( {1 \over b^{2}} + { 1 \over b^{3}} + {1 \over b^2} + { 1 \over b^{3} } +  { 1 \over b-1} + {1 \over (b -1)b^{2} } \right), $$
and therefore
\begin{equation} \vert  a - b^{y_3 - y_1} \vert \le b^{y_3-y_1} \left( {1 \over b-3} \right).  \label{121} \end{equation}
One can now derive 
\begin{equation} \left( 1 - {1 \over b-3} \right) b^{y_3-y_1} < a < \left( 1 + {1 \over b-3 } \right) b^{y_3-y_1}. \label{122} \end{equation}
If $y_3-y_1= 3$ or 4, then for each $b$ with $1000 < b < ( 8 \cdot 10^{14})^{{1 \over y_3-y_1}}$, each $y_1$ with $y_1 \le y_3-y_1$, and each $\beta, \gamma \in \{0,1\}$, we can calculate $y_3 = y_1 +(y_3-y_1)$ and then factor $b^{y_3} +(-1)^{\beta+\gamma}$.   Since $a^{x_2} \mid b^{y_3} +(-1)^{\beta+\gamma}$, we now consider each factor $a \mid b^{y_3} +(-1)^{\beta+\gamma}$ that satisfies (\ref{122}), and calculate the maximal value of $x_2$ such that $a^{x_2} \mid b^{y_3} + (-1)^{\beta+\gamma}$; call this $x_{2, max}$.  For each $x_2$ up to $x_{2,max}$, let $x_3 = x_2 + 1$.  For each $\alpha \in \{0,1\}$ and for the chosen $\beta$, set $r = (b^{y_1} - (-1)^{\beta})/h $, $s =   (a^{x_2} - (-1)^{\alpha})/h$, where $h = \gcd(a^{x_2} - (-1)^{\alpha}, b^{y_1} - (-1)^{\beta})$.   We must have $(-1)^\alpha r + s b^{y_1} = r a^{x_2} + (-1)^\beta s = (-1)^\gamma (r a^{x_3} - s b^{y_3})$.  Our calculations now show that the only possible sets of three solutions belong to the infinite class (\ref{10a}) with $x_2=1$; by Lemma~\ref{Case62Lemma}, none of these extends to four solutions.  

Consider now the case $x_3-x_2 = 1$ and $1 = y_1 < y_3-y_1=2$.  
We still have (\ref{122}) so $\left( 1 - {1 \over b-3 } \right) b^2 < a < \left( 1 + {1 \over b-3 } \right) b^2$.  Now $a^{x_2} \mid b^3 +(-1)^{\beta+\gamma}$ but $a^2 > \left( 1 - {1 \over b-3 } \right)^2 b^4 > b^3 + 1$, so $x_2=1$.  

If $\alpha \ne \gamma$ then, using (\ref{111}) and clearing denominators, we get  
$$ a (a +(-1)^{\alpha}) (b - (-1)^{\beta} ) = ( b^3 +(-1)^{\beta+\gamma}) ( a - (-1)^{\alpha}). $$
Since $\gcd( a - (-1)^{\alpha}, a + (-1)^{\alpha}) \le 2$, we must have $a-(-1)^{\alpha} \mid 2 (b - (-1)^{\beta})$ so $a \le 2b+3$.  This contradicts $\left( 1 - {1 \over b-3 } \right) b^2 < a$.  

If $\alpha = \gamma = 0$ then, using (\ref{111}) and dividing out $a-1$,  
\begin{equation}   a (b - (-1)^{\beta} ) = b^3 +(-1)^{\beta}. \label{X57X} \end{equation}
Since $\gcd(b - (-1)^{\beta}, b^3 +(-1)^{\beta} ) \le 2$ and $b-(-1)^\beta > 2$, this is impossible.  

Therefore $\alpha =\gamma=1$.  In this case, using (\ref{111}) as before and dividing out $a+1$, we get (\ref{X57X}), so 
$$ a =  b^2  +(-1)^{\beta} b + 1. $$
So we have a member of the infinite class (\ref{10a}) with $d=1$, $k=3$, $u=\beta+1$, $v=0$; by Lemma~\ref{Case62Lemma} this has no fourth solution.   
\end{proof}

\section{Case (\ref{20b}) $0 = x_1 < x_2 < x_3 < x_4$, $0 = y_1 = y_2 < y_3 < y_4$}    

Suppose (\ref{20b}) holds.  To fix notation, let $c = s - (-1)^\alpha r = r a^{x_2} - s = (-1)^\beta (r a^{x_3} - s b^{y_3}) = (-1)^\gamma (r a^{x_4} - s b^{y_4})$ for some $\alpha$, $\beta$, $\gamma \in \{ 0,1 \}$.  Considering (\ref{2}) with $(i,j)=(1,2)$, we see that $r (a^{x_2} + (-1)^\alpha) = 2s$, so $r=1$ or 2.  Applying (\ref{2}) with $(i,j)=(1,3)$ and $(2,3)$, we find $r(a^{x_3} + (-1)^{\alpha+\beta}) = s (b^{y_3} + (-1)^\beta)$ and $r a^{x_2} (a^{x_3-x_2} +(-1)^{\beta+1}) = s (b^{y_3} + (-1)^{\beta+1})$.  We find that, for any instance of case (\ref{20b}), we have  
\begin{equation} r \in \{1,2\}, 2 \mid a-r, s = { r(a^{x_2} +(-1)^\alpha) \over 2}, c= s +(-1)^{\alpha+1} r, b^{y_3} = { 2 ( a^{x_3}  +(-1)^{\alpha+\beta}) \over a^{x_2} + (-1)^\alpha }  - (-1)^\beta.   \label{J}  \end{equation}
Since $b^{y_3} $ is an integer, we can use the elementary divisibility properties of $a^{x} \pm 1$ to see that $x_2 \mid x_3$ when $a>3$.  And, since (\ref{J}) holds when $x_3$ is replaced by $x_i$ where $3 < i \le N$, we have, for $a>3$,
\begin{equation}  x_2 \mid x_i, 3 \le i \le N.  \label{H} \end{equation} 
Note that (\ref{J}) corresponds to the infinite family (\ref{67}).   

\begin{Lemma} 
\label{Small20Lemma}
Suppose $a \le 134000$.  Then (\ref{1}) with conditions (\ref{20b}) has at most three solutions.   
\end{Lemma}

\begin{proof}
Suppose (\ref{20b}) holds and $a \le 134000$.      

We first get bounds on $x_2$ and $x_3-x_2$.  Lemma~\ref{ZBoundLemma} gives $a^{x_3-x_2} \le Z$ so $x_3-x_2 <\log( 8 \cdot 10^{14}) / \log(a)$.  By Lemma~\ref{ZBoundLemma}, $s \le Z+1$ so $a^{x_2} \le 2(Z+1)+1 < 16 \cdot 10^{14}+3$, hence $x_2 < \log(16 \cdot 10^{14}+3) / \log(a)$.  
If $a \ge 5$ then $x_2 \mid x_3-x_2 < \log(8 \cdot 10^{14})/ \log(a)$.  

Given $a$, $x_2$, $x_3 - x_2$, and choosing $\alpha$, $\beta \in \{0,1\}$, and $r \in \{1,2\}$ with $2 \mid a-r$, we have (\ref{J}).  
Suppose for some choice of $a$, $\alpha$, $\beta$, $x_2$, and $x_3-x_2$, we find that ${ 2 ( a^{x_3}  +(-1)^{\alpha+\beta}) \over a^{x_2} + (-1)^\alpha }  - (-1)^\beta$ is an integer; it is $b^{y_3}$, so we know $b$ and its associated $y_3$ (without loss of generality we may assume $b$ is not a perfect power).  We now have three solutions to (\ref{1}) with this $a$, $b$, $c$, $r$, $s$ and choice of signs.  We use LLL (then bootstrapping if LLL fails) to show no fourth solutions exist.   A few cases with $a=2$ and $b= (2^{x} + 1)/3$ require a further elementary argument modulo 8 to eliminate a fourth solution.  
\end{proof}

\begin{Lemma} 
\label{Lemma16}
No instance of (\ref{20b}) has four or more solutions.  
\end{Lemma}

\begin{proof}
Suppose (\ref{20b}) holds with $N > 3$.  
By Lemma~\ref{Small20Lemma} we may assume $a>134000$ so that (\ref{H}) holds.  From (\ref{J}) we have $r=1$ if $a$ is odd, $r=2$ if $a$ is even.    
If $s \le 2$, then, considering the solution $(x_2, y_2)$,   we must have $a \le 5$, contradicting Lemma~\ref{Small20Lemma}.  So we can assume $s>2$.  We can also assume $r a^{x_2} > 2$.  And we certainly can assume that the solutions $(x_1, y_1)$, \dots, $(x_N, y_N)$ include all solutions to (\ref{1}) for $(a,b,c,r,s)$.  Now we can apply Lemma~\ref{NewLemma6} to the solutions $(x_2, y_2)$, $(x_3, y_3)$, $(x_4, y_4)$ to get 
\begin{equation}   y_3 \mid y_4.  \label{F} \end{equation}
Since we have $a>134000$, we see that $a^{x_3-x_2} > 8 \cdot 10^{14} $ when $x_3-x_2 \ge 3$, hence, by Lemma~\ref{ZBoundLemma}, we have $x_3-x_2 \le 2$.  
By (\ref{H}) we see that, in considering the set of solutions $(a,b,c,r,s; 0,0,x_2, 0, x_3, y_3, x_4, y_4)$, we can assume $x_2 = 1$ without loss of generality, noting that we are reformulating $a$, $x_3$ and $x_4$.  Now one can see that the only possibilities satisfying (\ref{J}) have $x_2 = 1$ with 

\begin{tabular}{c  c  c  c}
$x_3 - x_2$ & $\alpha$ & $\beta$ \\
2 & 0 & 0 \\
2 & 1 & 0 \\
1 & 0 & 1 \\
1 & 1 & 0 
\end{tabular}

\noindent
Note that $c \le r + s < 8 \cdot 10^{14}+2$ by Lemma~\ref{SummaryLemma}, and $\pm c = r a^{x_4} - s b^{y_4}$.  By Lemma~\ref{ZBoundLemma}, $134000< a < \max(x_4, y_4)$, so certainly $100 < x_4$.  Considering $\pm c = r a^{x_4} - s b^{y_4}$ modulo $a^{100}$, we have $s b^{y_4} \pm c \congruent 0 \bmod a^{100}$.  We will expand $y_4$ in powers of $a$; we have $y_4 < 8 \cdot 10^{14} < 134000^3 < a^3$, so we do not need powers of $a$ higher than two.  By (\ref{F}), we can assume $y_3 = 1$ without loss of generality, noting that in doing so we have reformulated $b$ and $y_4$.  Thus we are considering the following set of solutions: $(a,b,c,r,s; 0,0,1,0,x_3, 1, x_4, y_4)$, noting the reformulations of variables as indicated above.  

\noindent {\it Case 1:} 
Suppose $x_2=1$ and $x_3-x_2=2$.  Then $\beta=0$.  Using (\ref{J}) and taking the upper sign when $\alpha = 0$ and the lower sign when $\alpha = 1$, we have $b = 2 a^2 \mp 2 a + 1$ and, for either choice of the parity of $a$,
$$ (a \mp 1) + (a \pm 1) b^{y_4} \congruent (a \mp 1) +(a \pm 1) b \congruent 0 \bmod 2a^3 $$
so that $(a \pm 1) (b^{y_4 - 1} -1)  \congruent 0 \bmod 2a^3$, so that, for either choice of the parity of $a$, $b^{y_4-1} \congruent 1 \bmod 2a^3$ which requires $a^2 \mid y_4-1$, by Lemma 1 of \cite{Sc-St3}.  So, letting $y_4 = 1 + j a^2$ for some positive integer $j$, and letting $M_i$ be an integer for $1 \le i \le 3$, 
\begin{align}
&
\frac{a \mp 1}{2} +  \frac{a \pm 1}{2} (2 a^2 \mp 2a + 1)^{y_4}  \notag \\
&= \frac{a \mp 1}{2} +  \frac{a \pm 1}{2} \Bigg( 1 + y_4 (2 a^2 \mp 2a) + {y_4 (y_4 -1) \over 2} (2 a^2 \mp 2a)^2 + {y_4 (y_4 -1) (y_4 - 2) \over 6} (2 a^2 \mp 2a)^3  \notag \\
 & + {y_4 (y_4 -1) (y_4 - 2) (y_4 - 3) \over 24} (2 a^2 \mp 2a)^4 + 2 M_1 a^5 \Bigg) =   M_2 a^{100}
\label{234}
\end{align} 
since $x_4 > 100$.  
This simplifies to 
\begin{equation} (1 - j)a^3 \pm j a^4 + \frac{M_3}{3} a^5 = M_2 a^{100}.  \label{J_5} \end{equation} 
From this we see that $j = 1 + wa/3$ for some integer $w \ge 0$.  If $w=0$ then (\ref{J_5}) becomes impossible.  And if $w>0$, then $y_4 > \frac{a}{3} a^2 > 8 \cdot 10^{14}$ when $a>134000$, giving a contradiction to Lemma~\ref{SummaryLemma}.

\noindent {\it Case 2:}
Suppose $x_2 = x_3 - x_2=1$, so that $\alpha \ne \beta$ and $b = 2 a - (-1)^\alpha$.  Considering the solution $(x_3, y_3)=(2,1)$, 
we have, for either choice of the parity of $a$, 
\begin{equation}  (a-(-1)^\alpha) - (-1)^\alpha (a+(-1)^\alpha) b \congruent 0 \bmod 2 a^2.  \label{MS1} \end{equation} 
Let $t=1$ if $y_4$ is odd and $\alpha = 0$, otherwise let $t=0$.  Then considering the solution $(x_4, y_4)$ we have 
\begin{equation}  (a - (-1)^\alpha ) + (-1)^t (a + (-1)^\alpha) b^{y_4} \congruent 0 \bmod 2 a^{100} \label{MS2} \end{equation} 
since $x_4 > 100$. Combining (\ref{MS1}) and (\ref{MS2}) we have
$(a+(-1)^\alpha) (b^{y_4-1} + (-1)^{t+\alpha}) \congruent 0 \bmod 2 a^2$, which requires $a \mid y_4-1$ for either choice of the parity of $a$, by Lemma 1 of \cite{Sc-St3}.  
So we can write $y_4 = 1 + e_1 a + e_2 a^2$ where $0 \le e_1 < a$.  

Assume first $\alpha = 0$.  Let $M_1$ and $M_2$ be integers.  Then using (\ref{MS2}) we have 
\begin{equation}   \frac{a-1}{2} + \frac{a+1}{2} \left( 1 - y_4 (2a) + \frac{y_4(y_4-1)}{2} (2a)^2 - \frac{y_4(y_4-1)(y_4-2)}{6} (2a)^3 + 2 M_1 a^4 \right) = M_2 a^{100}. \label{TH} \end{equation} 
This simplifies to 
$$ (-1 -e_1) a^2 = M_3 a^3 $$
for some integer $M_3$.  Thus $e_1 = a-1$.  So $y_4 = 1 + (a-1)a + e_2 a^2$.  Now from (\ref{TH}) we obtain
$$ (-1-e_2) a^3 = \frac{M_4}{3} a^4 $$
for some integer $M_4$, so that $e_2 = ({wa}/{3}) - 1$ for some integer $w \ge 1$.   But then $y_4 \ge 1 + (a-1) a + \left( \frac{a}{3} -1 \right) a^2   > 8 \cdot 10^{14}$ when $a> 134000$, violating Lemma~\ref{SummaryLemma}.  

So we can assume $\alpha = 1$.   We define $M_i$, $1 \le i \le 6$ to be integers.   From (\ref{MS2}) we derive 
$$ \frac{a+1}{2} + \frac{a-1}{2} \left( 1 + y_4 (2a) + \frac{y_4(y_4-1)}{2} (2a)^2 + 2 M_1 a^3 \right) = M_2 a^{100}.  $$
This simplifies to 
$$ (1-e_1) a^2 = M_3 a^3. $$
Thus, $e_1 = 1$.  Write $y_4 = 1 + a + e_2 a^2$.  Then
\begin{align*}
 \frac{a+1}{2} + \frac{a-1}{2} \Bigg( 1 + y_4 (2a) + \frac{y_4(y_4-1)}{2} (2a)^2 + \frac{y_4(y_4-1)(y_4-2)}{6} (2a)^3 & \notag \\
+ \frac{y_4(y_4-1)(y_4-2)(y_4-3)}{24} (2a)^4 + 2 M_4 a^5 \Bigg) & = M_5 a^{100}. \\
\end{align*}
This simplifies to 
\begin{equation}  -e_2 a^3 + \frac{2}{3} a^4 = \frac{M_6}{3} a^5 \label{MS3}  \end{equation}    
so that $e_2 = {wa}/{3}$ for some integer $w \ge 0$.  If $w=0$ then (\ref{MS3}) becomes impossible.  
And if $w>0$, $y_4 \ge 1+a +\frac{a}{3} a^2 > 8 \cdot 10^{14}$ since $a> 134000$, violating Lemma~\ref{SummaryLemma}.  

This completes the proof of Lemma~\ref{Lemma16}.  
\end{proof}

Comment on Theorem 2 of \cite{Sc-St4}: If (\ref{401}) holds, then, by Lemma~\ref{x2y1BoundLemma}, we can assume without loss of generality that $y_1 \le y_3-y_1$.  If $N > 3$, applying Lemma~\ref{ZBoundLemma} we get  
\begin{equation} c \le r + s b^{y_1} \le r + s b^{y_3-y_1} \le (Z+1) + (Z+1) Z = (Z+1)^2. \label{33new} \end{equation}
It follows from the proof of Theorem 2 of \cite{Sc-St4} (Case 2, equation (76a)), that we must have 
\begin{equation} Z < 0.9 + \frac{\log(c) }{\log(2)} + 1.6901816335 \cdot 10^{10} \cdot \log(Z+1) \log(Z) \log(1.5 e Z). \label{34new} \end{equation}
Using (\ref{33new}) in (\ref{34new}) we get $Z < 7.07 \cdot 10^{14}$, completing a simpler and shorter proof of Theorem 2 of \cite{Sc-St4}.

\end{document}